\documentclass[12pt,twoside]{amsart}
 
\usepackage{amsmath} 
\usepackage{amssymb} 
 \usepackage{color}

\newtheorem{thm}{Theorem} 
 
\newtheorem{lemma}[thm]{Lemma} 
 
\newtheorem{defi}[thm]{Definition} 
\newtheorem{preremark}[thm]{Remark} 
\newenvironment{remark}{\begin{preremark}\rm}{\end{preremark}} 
\def\R{{\mathbb R}} 
\def\a{{\alpha}}

\title[]{A Liouville theorem for non local elliptic equations }
 
\author[L. Dupaigne]{Louis Dupaigne} 
\author[Y. Sire]{Yannick Sire}

\begin{document} 
\begin{abstract} 
We prove a Liouville-type theorem for bounded stable solutions $v~\in~C^2(\R^n)$ of elliptic equations of the type 
\begin{equation*} 
(-\Delta)^s v= f(v)\qquad 
{\mbox{ in $\R^n$,}} 
\end{equation*} 
where $s \in (0,1)$ { and $f$ is any nonnegative function}. The operator $(-\Delta)^s$ stands for the fractional 
Laplacian, a pseudo-differential operator of symbol $|\xi |^{2s}$.  

 \end{abstract} 
 
\maketitle 
 
 
\noindent{\em Keywords:} Boundary reactions, 
fractional operators. 
\bigskip 
 
\noindent{\em 2000 Mathematics Subject Classification:} 
35J25, 47G30, 35B45, 53A05. 
 
\section*{Introduction} 
 
This paper is devoted to the proof of a Liouville theorem for stable solutions $v\in C^2(\R^n)$ of   
\begin{equation}\label{lapFrac0} (-\Delta)^s v= f(v)\qquad {\mbox{ in  
$\R^n$,}} \end{equation} 
where  $n\ge1$ and $f\in C^{1,\beta}(\R)$ is a nonnegative function for some $\beta \in (0,1)$. 
Given $s\in(0,1)$, the operator $(-\Delta)^s$ is the  
fractional  
Laplacian and it  is defined in various ways, which we review now. 
\subsection*{A quick review of the fractional Laplacian}
\begin{defi}\label{defi1} The fractional Laplacian is defined 
for $v\in H^s(\R^n)$ by
$$
\mathcal F((-\Delta)^s v) = \left| \xi \right|^{2s}\mathcal F(v), 
$$
where $\mathcal F$ denotes the Fourier transform.
\end{defi}
The fractional Laplacian is a nonlocal operator, as can be seen by taking inverse Fourier transforms in the above formula. We obtain the equivalent definition (see \cite{landkof} for a proof):
\begin{defi}\label{defi2} For all $x\in\R^n$, 
\begin{equation}\label{kernel} 
(-\Delta)^s v (x)= C_{n,s} P.V. \int_{\R^n}  
\frac{v(x)-v(t)}{|x-t|^{n+2s}}\,dt, 
\end{equation} 
where $C_{n,s}$ is a normalizing constant, where $P.V.$ stands for  
the Cauchy principal value and where $v$ is taken e.g. in $\mathcal S(\R^n)$ in order to define the (singular) integral in the usual sense. 
\end{defi} 
This nonlocal character makes the analysis of 
equations such as \eqref{lapFrac0} more difficult.  However, it is a well-known fact 
in harmonic analysis that for the power $s=1/2$, the fractional Laplacian can be realized as the boundary operator of harmonic functions in the half-space (see \cite{stein}).  Such a realization can be extended to general $s\in(0,1)$
as follows.

Given $s\in(0,1)$, let $\alpha=1-2s\in(-1,1)$. Using variables $(x,y)\in \R^{n+1}_+ 
:=(0,+\infty)\times\R^n$, the space $H^s(\R^n)$ coincides with the trace on $\partial\R^{n+1}_{+}$ of 
$$H^1(x^\alpha):= \left\{u\in H^1_{loc}(\R^{n+1}_{+})\;:\;\int_{\R^{n+1}_{+}}x^\alpha\left(u^2+\left| \nabla u \right|^2 \right) dxdy<+\infty\right\}.$$  
In other words, given any 
function $u\in H^1(x^\alpha)\cap C(\overline{\R^{n+1}_{+}})$, $v:=\left. u \right|_{\partial\R^{n+1}_{+}}\in H^s(\R^n)$, and there exists a constant $C=C(n,s)>0$ such that
$$
\| v\|_{H^s(\R^n)}\le C \| u \|_{ H^1(x^\alpha)}.  
$$
So, by a standard density argument (see \cite{CPSC}), every $u\in H^1(x^\alpha)$ has a well-defined trace $v\in H^s(\R^n)$.
Conversely, any $v\in H^s(\R^n)$ is the trace of a function $u\in H^1(x^\alpha)$. In addition, the function $u\in H^1(x^\a)$ defined by
\begin{equation} \label{argmin} 
u:=\arg\min\left\{ \int_{\R^{n+1}_{+}}x^\a \left| \nabla w \right|^2\;dx \; : \; \left. w \right|_{\partial\R^{n+1}_{+}}=v\right\} 
\end{equation} 
solves the PDE 
\begin{equation}\label{bdyFrac2} 
\left \{
\begin{aligned} 
{\rm div}\, (x^\a \nabla u)&=0 \qquad 
{\mbox{ in $\R^{n+1}_+$}} 
\\
u&= v  
\qquad{\mbox{ on $\partial\R^{n+1}_+$.}}\end{aligned}\right . \end{equation} 
By standard elliptic regularity, $u$ is smooth in $\R^{n+1}_{+}$. It turns out that 
$x^\alpha u_{x} (x,\cdot)$ converges in $H^{-s}(\R^n)$ to a distribution $f\in H^{-s}(\R^n)$, as $x\to 0^+$ i.e. $u$ solves
\begin{equation}\label{bdyFrac3} 
\left \{
\begin{aligned} 
{\rm div}\, (x^\a \nabla u)&=0 \qquad 
{\mbox{ in $\R^{n+1}_+ 
$}} 
\\
-x^\a u_x &= f  
\qquad{\mbox{ on $\partial\R^{n+1}_+$.}}\end{aligned}\right . \end{equation}
Consider
the Dirichlet-to-Neumann operator 
$$
\Gamma_\a: 
\left\{
\begin{aligned}
H^s(\R^n)&\to H^{-s}(\R^n)\\
v&\mapsto \Gamma_{\a}(v)=  f:=
-x^\a u_x|_{\partial \R^{n+1}_+}, 
\end{aligned}
\right.
$$
where $u$ is the solution of \eqref{argmin}--\eqref{bdyFrac3}. Then, 
\begin{defi} \label{defi3}
There exists a constant $d_{n,s}>0$ such that for every $v\in H^s(\R^n)$,
\begin{equation*} 
(-\Delta)^s v = d_{n,s}\Gamma_{\a}(v),
\end{equation*} 
where $\a=1-2s$.
\end{defi}
\noindent In other words, given $f\in H^{-s}(\R^n)$, a function $v\in H^s(\R^n)$ solves the equation 
\begin{equation} \label{linear} 
\frac1{d_{n,s}}(-\Delta)^{s} v=f\qquad\mbox{in $\mathbb R^n$}
\end{equation} 
if and only if its lifting $u\in H^1(x^\a)$ solves $u= v{\mbox{ on $\partial\R^{n+1}_+$}}$ and
\begin{equation}\label{lifted}  
\left \{
\begin{aligned} 
{\rm div}\, (x^\a \nabla u)&=0 \qquad 
{\mbox{ in $\R^{n+1}_+$}} 
\\
-x^\a u_x &= f  
\qquad{\mbox{ on $\partial\R^{n+1}_+.$}}
\end{aligned}\right. 
\end{equation} 
For a proof of the claims that lead us to Definition \ref{defi3}, we refer the reader to \cite{cafS}.

\noindent Observe than none of the definitions \ref{defi1},\ref{defi2}, \ref{defi3}  give a proper way of defining $(-\Delta)^s v$ for arbitrary $v\in C^2(\R^n)$. 
However, Definitions \ref{defi2} and \ref{defi3}  can be extended to the class of \it bounded \rm functions $v\in C^2(\R^n)$ and they coincide, using the following results due to \cite{cafS}.
\begin{lemma}\label{poisson}
The Poisson kernel $P$ defined for $(x,y)\in\R^{n+1}_{+}$ by
\begin{equation}\label{Poisson}
P(x,y)=c_{n,\a}\frac{x^{1-\a}}{\Big
( x^2+|y|^2\Big)^\frac{n+1-\a}{2}}
\end{equation}
is a solution of
\begin{equation}\label{POI}
\left\{ \begin{aligned}
-{\rm div}\, (x^\a \nabla P)&=0 \qquad
{\mbox{ in $\R^{n+1}_{+},$}}
\\ P&=\delta_0
\qquad{\mbox{on $\partial\R^{n+1}_+$,}}\end{aligned}
\right.\end{equation}
where $\a \in (-1,1)$ and $c_{n,\a}$ is a normalizing constant 
such that
$$\int_{\R^n} P(x,y)\,dy=1,\qquad\text{for all $x>0$}. $$
\end{lemma}

\begin{thm}\label{CSPoisson}
Let $v\in C^2(\R^n)$ denote a bounded solution of \eqref{linear} (where $(-\Delta)^s v$ is given by Definition \ref{defi2}). Then, 
$$u=P \, *\, v$$
is the unique bounded weak solution of \eqref{lifted}, in the sense of the definition below. 
\begin{defi}
We say that $u \in L^\infty_{\rm loc}(\overline{\R^{n+1}_+})$ is a weak solution of \eqref{lifted} if 
\begin{equation}\label{hgasj7717177} 
x^\alpha |\nabla u|^2 \in L^1 (B_R^+) 
\end{equation} 
for any $R>0$, 
and if\footnote{Condition \eqref{hgasj7717177} 
is assumed here to make sense of \eqref{eq1}. It is 
always uniformly fulfilled when $u$ is bounded according to the regularity theory developped in \cite{CS}}
\begin{equation}\label{eq1} 
\int_{{\R^{n+1}_+}} x^\alpha \nabla u\cdot 
\nabla \varphi\;dx\;dy= 
\int_{\partial {\R^{n+1}_+}} 
f\varphi\;dy
\end{equation} 
for any $\varphi:\R^{n+1}_+\rightarrow \R$ which is bounded, locally
Lipschitz in the interior of
$\R^{n+1}_+$,
which 
vanishes on $\R^{n+1}_+\setminus B_R$ and such that
\begin{equation}\label{hgasj7717177-bis}
x^\a|\nabla\varphi|^2
\in L^1 (B_R^+).\end{equation}
\end{defi}
\end{thm}
We note that Theorem \ref{CSPoisson}  has been recently used  to prove full regularity of the solutions of the quasigeostrophic model as given by~\cite{CV}
and in free boundary analysis in \cite{CSS}. Also, several works have been devoted to equations of the type \eqref{bdyFrac}, starting with the pioneering work of Cabr\'e and Sola-Morales where they investigate the case $\alpha=0$  (see \cite{CSM}). One of the authors and Cabr\'e have extended their techniques to any power $\alpha\in(-1,1)$ (see \cite{CS}). 

To complete our review, we mention the probabilistic point of view: the fractional Laplacian can be seen as the infinitesimal generator of a Levy process (see, e.g.,~\cite{B}). 
This type of diffusion operators also arises in several areas such as  
optimization~\cite{DL}, flame propagation~\cite{CRS} 
and finance~\cite{CT}. Phase transitions driven by
fractional Laplacian-type boundary effects have also been
considered in \cite{ABS} in the Gamma convergence
framework. Power-like nonlinearities for boundary reactions
have been studied in \cite{POW}.

\subsection*{The boundary reaction problem}

We begin now our investigation of bounded solutions $v\in C^2(\R^n)$ of \eqref{lapFrac0}. For notational convenience, we actually study
\begin{equation}\label{lapFrac}  
\frac1{d_{n,s}}(-\Delta)^{s} v = f(v)\qquad\text{in $\R^n$}
\end{equation} 
and its equivalent formulation
\begin{equation}\label{bdyFrac}
\left \{
\begin{aligned} 
{\rm div}\, (x^\a \nabla u)&=0 \qquad 
{\mbox{ in $\R^{n+1}_+$}} 
\\
-x^\a u_x &= f(u)  
\qquad{\mbox{ on $\partial\R^{n+1}_+.$}}
\end{aligned}\right. 
\end{equation} 

We will concentrate on particular solutions of \eqref{lapFrac} and \eqref{bdyFrac} which are stable in the following sense. 
\begin{defi}
A bounded solution $v\in C^2(\R^n)$ of \eqref{lapFrac} is stable if for all { $\psi \in 
H^s(\R^n)$} we have
$$\frac1{d_{n,s}}\int_{\R^n} |(-\Delta)^{\frac{s}{2}} \psi|^2\;dy -\int_{\mathbb R^n} f'(v) \psi^2\;dy \geq 0. $$ 
\end{defi}

\begin{defi}
A bounded weak solution $u$ of \eqref{bdyFrac} is stable if  
\begin{equation}\label{sta1} 
\int_{\R^{n+1}_+} x^\alpha |\nabla\varphi|^2\;dx\;dy-\int_{\partial\R^{n+1}_+} 
f'(u)\varphi^2\;dy\,\ge\,0 
\end{equation} 
for any { $\varphi\in H^1(x^\a)$.}
\end{defi}
Note that $v$ is stable if and only if its lifting $u=P*v$ is stable.  
Note also that the stability assumption (which is best seen in the sense \eqref{sta1}) is satisfied by two interesting classes of solutions : monotone solutions and local minimizers (see e.g. \cite{CS}, \cite{SV}).

We prove the following results. 
 
  \begin{thm}\label{liouvilleFrac}
 Let  $\beta \in (0,1)$ and let $f$ be a { $C^{1,\beta}(\R)$} function such that $f \geq 0$. Let $v\in C^2(\R^n)$ denote a {bounded} stable solution of \eqref{lapFrac}. 
Then, we have: 
\begin{itemize}
\item Let $s \in [\frac{1}{2},1]$. Then $v$ is constant whenever $n \leq 3$.
\item Let $s \in (0,\frac{1}{2})$. Then $v$ is constant whenever $n \leq 2$.   
\end{itemize} 
\end{thm}
 
 The previous theorem is actually a corollary of the following result, applying to equation \eqref{bdyFrac}. 
 
 \begin{thm}\label{liouvilleBdy}
Let  $\beta \in (0,1)$ and let $f$ be a { $C^{1,\beta}(\R)$} function such that $f \geq 0$.  Let $u$ be a bounded stable weak solution of \eqref{bdyFrac}. Then we have: 
 \begin{itemize}
 \item Let $s \in [\frac{1}{2},1]$. Then $u$  is constant whenever $n \leq 3$.
 
\item Let $s \in (0,\frac{1}{2})$. Then $u$ is constant whenever $n \leq 2$.   
 \end{itemize} 
 \end{thm}
 
 We do not know whether Theorem \ref{liouvilleBdy} holds for $n=4$.  We note that for the standard Laplacian (case $s=1$), the theorem is true at least up to dimension $n=4$ (see \cite{DF}).  
 
 \section{Preliminary results}
 
 In this section, we give some preliminary results on the boundary problem \eqref{bdyFrac} for $n=1$, namely 
 
 \begin{equation}\label{1D}
\left \{
\begin{matrix} 
{\rm div}\, (x^\a \nabla u)=0 \qquad 
{\mbox{ on $\R^{2}_+ $}} 
\\
-x^\a \partial_x  u= f(u)  \geq 0. 
\qquad{\mbox{ in $\partial\R^{2}_+ $,}}\end{matrix}\right . \end{equation} 

We first state a boundary version of a well-known Liouville theorem of Berestycki, Caffarelli and Nirenberg (see \cite{BCN}). The following result is proved in \cite{CS} (see also \cite{CSM}). We include the proof here for the sake of completeness.  

\begin{thm}(\cite{CS})\label{liouville}
Let $\varphi \in L^{\infty}_{loc}(\mathbb{R}_+^{n+1})$ be a positive function. Suppose that $\sigma \in
H^1_{loc}(\mathbb{R}_+^{n+1})$, that 
$$x^\alpha |\nabla \sigma |^2 \in L_{loc}^1(\R^{n+1}_+)$$
and that $\sigma$ solves 
\begin{equation}\label{liouveq}
\begin{cases}
{\rm div}(x^\alpha \varphi^2 \nabla\sigma) = 0
&\quad \hbox{in } \R^{n+1}_+\\
-x^\alpha \partial_x \sigma  \leq 0
&\quad \hbox{on } \partial\R^{n+1}_+
\end{cases}
\end{equation}
in the weak sense. Assume that for every $R >1$, 
\begin{equation}\label{intR2}
\int_{B^+_R} x^{\alpha} (\sigma \varphi)^2 \;dxdy\leq C R^2
\end{equation}
for some constant $C$ independent of $R$. 

Then $\sigma$ is constant. 
\end{thm}
\begin{proof}
We adapt the proof given in \cite{CSM}. Let $\zeta$ be a $C^\infty$ function on $\R^+$ such that
$0\le\zeta\le 1$ and  
$$
\zeta = \begin{cases}  1&\quad\text{for $0\le t\le 1$,}\\
                0&\quad\text{for $t\ge 2$.}
\end{cases}
$$
For $R>1$ and $(x,y)\in\R^{n+1}_{+}$, let
$\zeta_R(x,y)=\zeta \left( r / R\right)$, where $r=\vert(x,y)\vert$.

Multiplying \eqref{liouveq} by $\zeta_R^2$ and integrating by parts
in $\R^{n+1}_+$, we obtain
\begin{equation*}
\begin{split}
&\int_{\R^{n+1}_+} x^\alpha \zeta_R^2\, \varphi^2 \vert\nabla\sigma\vert^2\;dxdy \leq 
2 \int_{\R^{n+1}_+}x^\alpha \zeta_R\, \varphi^2 \sigma\, \nabla\zeta_R\, \nabla\sigma \;dxdy\\
& \le 2 \left [ \int_{\R^{n+1}_+\cap\{R<r<2R \}} 
\zeta_R^2\, \varphi^2 x^\alpha \vert\nabla\sigma\vert^2\;dxdy \right ]^{1/2}
\left [ \int_{\R^{n+1}_+} \varphi^2 x^\alpha \sigma^2 \vert\nabla\zeta_R\vert^2 \;dxdy
\right ]^{1/2} \\
& \le C \left [ \int_{\R^{n+1}_+\cap\{R<r<2R \}} 
\zeta_R^2\, \varphi^2 x^{\alpha} \vert\nabla\sigma\vert^2 \;dxdy\right ]^{1/2}
\left [\frac{1}{R^2} \int_{B_{2R}^+} x^{\alpha} (\varphi\sigma)^2 \;dxdy
\right ]^{1/2} ,    
\end{split} 
\end{equation*}
for some constant $C$ independent of $R$. Using hypothesis
\eqref{intR2}, we infer that 
\begin{multline}\label{boundann}
\int_{\R^{n+1}_+}x^\alpha \zeta_R^2\, \varphi^2 \vert\nabla\sigma\vert^2\;dxdy \le \\
C \left [ \int_{\R^{n+1}_+\cap\{R<r<2R \}} 
\zeta_R^2\, \varphi^2 x^{\alpha} \vert\nabla\sigma\vert^2\;dxdy \right ]^{1/2} ,
\end{multline}
again with $C$ independent of $R$. Hence,
$\int_{\R^{n+1}_+} \zeta_R^2\, \varphi^2 x^{\alpha} \vert\nabla\sigma\vert^2\;dxdy \le C$ and,
letting $R\rightarrow\infty$, we deduce
$\int_{\R^{n+1}_+}\varphi^2 x^\alpha \vert\nabla\sigma\vert^2\;dxdy \le C$.
It follows that the right hand side of \eqref{boundann} tends to
zero as $R\rightarrow\infty$, and therefore
$\int_{\R^{n+1}_+} x^\alpha \varphi^2 \vert\nabla\sigma\vert^2\;dxdy =0$.
We conclude that $\sigma$ is constant.
\end{proof}

A second important lemma is the following, which can also be found in \cite{CS} (see also \cite{CSM}). 
\begin{lemma}(\cite{CS})\label{polipo}
Let $d$ be a bounded, H\"older continuous 
function on $\partial\R^{2}_+$. Then,
\begin{equation}\label{stablebis}
\int_{\R^{2}_+} x^\alpha |\nabla\xi|^2 \;dxdy+\int_{\partial\R^{2}_+}
d(y)\xi^2\;dy\ge 0
\end{equation}
for every function $\xi\in C^1(\overline{\R^{2}_+})$ with compact support in 
$\overline{\R^{n+1}_+}$, if and only if 
there exists a function $\varphi $ such that $\varphi >0$ in $\overline{\R^{2}_+}$ and
\begin{equation} \label{linear}
\begin{cases}
\mbox{div}\, (x^\alpha \nabla \varphi) = 0&\text{ in } \R^{2}_+\\ 
-x^\alpha \dfrac{\partial\varphi}{\partial x}+d(y)\varphi = 0&\text{ on }
\partial\R^{2}_+ .
\end{cases}
\end{equation}
\end{lemma}

We can prove now the following theorem. 

\begin{thm}(\cite{CS})\label{sign}
Let $u$ be a stable bounded weak solution of \eqref{1D}. Then, 
\begin{itemize}
\item either $u_y >0$ in $\overline{\R^2_+}$,
\item either $u_y <0$ in $\overline{\R^2_+}$,
\item or $u_y \equiv 0$ in $\overline{\R^2_+}$.
\end{itemize} 
\end{thm}
\begin{proof}
The proof is already contained in \cite{CS} (see also \cite{CSM}) but we reproduce it here for sake of completeness.

Since $u$ is assumed to be a stable solution, then \eqref{stablebis}
holds with $d(y)=-f'(u(0,y))$. 
Hence, by Lemma~\ref{polipo}, there exists a function
$\varphi >0$ in $\overline{\R^{2}_+}$ such that
\begin{equation*} \label{linearf}
\begin{cases}
\mbox{div}\,(x^\alpha \nabla  \varphi) = 0&\text{ in } \R^{2}_+\\ 
-x^\alpha \dfrac{\partial\varphi}{\partial x }-f'(u(0,y))\varphi = 0&\text{ on }
\partial\R^{2}_+ .
\end{cases}
\end{equation*}

Consider the function
$$
\sigma =\frac{u_{y}}{\varphi}.
$$
It is easy to check that  
$$
{\rm div}(x^\alpha \varphi^2 \nabla\sigma)=0 \quad\; \mbox{in }\R^{2}_+ 
$$
and $-x^\alpha \frac{\partial \sigma}{\partial x}=0$ on
$\partial\R^{2}_+$. We can use the Liouville Theorem  \ref{liouville},
and deduce that $\sigma$ is constant, provided
$$
\int_{B_R^+} x^\alpha (\varphi\sigma)^2 \;dxdy \leq CR^2
\quad\quad\mbox{ for all } R>1 ,
$$
holds for some constant $C$ independent of $R$. 
But note that
$\varphi\sigma=u_{y}$, and therefore
$$
\int_{B_R^+} x^\alpha (\varphi\sigma)^2\;dxdy \le
\int_{B_R^+} x^\alpha |\nabla u|^2\;dxdy  
$$
and we end up estimating the last inequality. Using the weak formulation \eqref{eq1} with the test function $u
\tau ^2$ where $\tau$ is a cutoff
function
such that $0\le\tau\in C^1_c (B_{2R})$, with $\tau=1$ 
in $B_{R}$ and $|\nabla \tau|\le 8/R$, with $R\ge1$, one gets that 
\begin{eqnarray*}&& 
\int_{{\R^{2}_+}}
x^\alpha\,
\big( |\nabla u |^2 \tau ^2+2 \tau \nabla u \cdot \nabla \tau 
\big)\;dxdy=
\int_{\partial \R^{2}_+} f(u) u\tau^2\;dy. 
\end{eqnarray*} 
Thus, by the Cauchy-Schwarz
inequality,
\begin{eqnarray*}
\int_{\R_+^{2}}x^\alpha\,|\nabla u|^2\tau^2\;dxdy
&\le& \frac 12 \int_{\R_+^{2}}x^\alpha\,|\nabla u|^2\tau^2\;dxdy,
\\&&\quad
+
C_* \Big(
\int_{\R_+^{2}}
x^\alpha |\nabla \tau|^2\;dxdy
+\int_{\R}|f(u)|\,|u|\,\tau^2\;dy
\Big)\\
\end{eqnarray*}
for a suitable constant $C_*>0$.

Since $u$ is bounded and $f$ is $C^{1,\beta}$, one gets the desired result. As a consequence, we have 
$$u_y=c \varphi $$
and depending on the sign of the constant $c$, this gives the desired conclusion. 
\end{proof}

\section{Proof of Theorem \ref{liouvilleBdy}}

As previously described, the important step is to get a Liouville theorem for the boundary problem \eqref{bdyFrac}. We adopt the method developed in \cite{DF}, based on choosing suitable test functions in the weak formulation  
\begin{equation}
\int_{{\R^{n+1}_+}} x^\alpha \nabla u\cdot 
\nabla \varphi\;dx\;dy= 
\int_{\partial {\R^{n+1}_+}} 
f(u)\varphi\;dy
\end{equation}

We first prove the following energy bound. 

\begin{lemma}
Let $u$ be a bounded weak solution of  \eqref{bdyFrac}. Then, there exists a constant $C>0$ such that for all $R>1$ 
\begin{equation}\label{energ}
\int_{B^+_R} x^\alpha |\nabla u|^2\;dx\;dy \leq C R^{n+\alpha-1}. 
\end{equation} 
\end{lemma}
\begin{proof}
Let $M=\sup_{\overline{\R^{n+1}_+}} u$. { Given $R>1$ and a cut-off function $\psi_{1}\in C^2_{c}(\R^{n+1})$ such that $\psi_{1}(z)=1$ on $|z|\leq 1$ and $\psi_{1}(z)=0$ on $|z|\geq 2$, let $\psi_R(x)=\psi_{1}(x/R)$.} We choose 
$$\varphi=(u-M) \psi_R. $$
This leads to 
$$\int_{\R^{n+1}_+} x^\alpha (u-M)\nabla u \cdot \nabla \psi_R \;dx\;dy+ \int_{\R^{n+1}_+} x^\alpha \psi_R |\nabla u|^2\;dx\;dy=$$
$$\int_{\R^n} f(u) (u-M) \psi_R \;dy\leq 0$$
since $f$ is nonnegative. 
Hence we have 
$$\int_{\R^{n+1}_+} x^\alpha \psi_R |\nabla u|^2 \;dx\;dy\leq -\int_{\R^{n+1}_+} x^\alpha (u-M)\nabla u \cdot \nabla \psi_R \;dx\;dy$$
and we are left to estimate the right hand side. Performing an integration by parts gives 
$$\int_{\R^{n+1}_+} x^\alpha (u-M)\nabla u \cdot \nabla \psi_R \;dx\;dy= \frac{1}{2}\int_{\R^{n+1}_+} x^\alpha \nabla (u-M)^2 \cdot \nabla \psi_R\;dx\;dy=$$
$$-\frac{1}{2}\int_{\R^{n+1}_+} (u-M)^2 \nabla \cdot (x^\alpha \nabla \psi_R)\;dx\;dy +\frac{1}{2}\int_{\R^n}(u-M)^2 (x^\alpha \partial_x \psi_R)|_{x=0}\;dy= $$
$$I +II. $$

We have 
$$I=-\frac{1}{2}\int_{\R^{n+1}_+} (u-M)^2 \Big \{ x^\alpha \Delta \psi_R + \alpha x^{\alpha-1} \partial_x \psi_R \Big \}\;dx\;dy. $$

Since $|\nabla \psi_R| \leq C/R$ and $|\Delta \psi_R | \leq C/R^2$ on $B^+_{2R}$, we are led to
$$I \leq C M^2 \int_{B_{2R}^+} \left(\frac{x^\alpha}{R^2} + \frac{x^{\alpha-1}}{R}\right)\;dx\;dy, $$
Recall that $\psi_R \equiv 0$ on $B_R^+$. Then,
$$I \leq C M^2 \int_{R}^{2R} dx \int_{B'_R}dy \Big \{  \frac{x^\alpha}{R^2} + \frac{x^{\alpha-1}}{R} \Big \},  $$
where $B'_R$ is the unit ball of radius $R$ in $\R^n$. Hence 
$$I \leq C M^2 R^n \int_{R}^{2R} dx  \Big \{  \frac{x^\alpha}{R^2} + \frac{x^{\alpha-1}}{R} \Big \}  $$
and then 
$$I \leq C M^2 R^n \big \{ \frac{R^{\alpha+1}}{R^2} + \frac{R^{\alpha}}{R} \Big \} \leq C R^{n+\alpha -1}.  $$
We now come to the estimate of the term $II$. { By definition of $\psi_R$, there exists a constant $C>0$ such that $\left| x^\alpha\partial_{x}\psi_R(x,y) \right|\le CR^{\alpha-1}\chi_{B_{2R}}$ for all $(x,y)\in\R^{n+1}_{+}$. It follows that
$$II \leq C R^{n+\alpha-1 },$$
as desired.} 

\end{proof}

The next theorem is proved in \cite{SV} (see also \cite{CS}). 
\begin{thm}(\cite{SV},\cite{CS})\label{aux:P} 
Let $u$ be a stable bounded weak solution of \eqref{bdyFrac}.
Assume furthermore that  
there exists~$C_0\geq 1$ such that 
\begin{equation}\label{en:bound} 
\int_{B^+_R} x^\alpha |\nabla u|^2\;dx\;dy\le C_0\,  
R^2\end{equation} 
for any~$R\ge C_0$. 
 
Then there exist~$\omega \in \mathbb S^{n-1}$ 
and~$u_0: (0,+\infty) \times \R \rightarrow\R$ 
such that 
\begin{equation}\label{31bis}
u(x,y)=u_0(x,\omega 
\cdot y)\end{equation}
for any~$(x,y)\in\R^{n+1}_+$.
\end{thm}

We now can proceed to the proof of Theorem \ref{liouvilleBdy}. From Theorem \ref{aux:P}, we conclude that if $\alpha \in (-1,0]$ and $n\le3$, or if  $\alpha \in (0,1)$ and $n\le2$, then $u$ is of the form $u_0(x,\omega \cdot y)$. The function $u_0$ is bounded, stable  and satisfies (in the weak sense) 
\begin{equation}
\left \{
\begin{matrix} 
{\rm div}\, (x^\a \nabla u_0)=0 \qquad 
{\mbox{ on $\R^{2}_+ 
:=\R\times(0,+\infty)$}} 
\\
-x^\a \partial_x u_0 = f(u_0)  
\qquad{\mbox{ on $\R\times\{0\}$.}}\end{matrix}\right . \end{equation} 

From Theorem \ref{sign}, we then have that either $\partial_y u_0 \equiv 0$ or $u_0$ is strictly monotone  in $y$ in $\overline{\R^2_+}$. If $\partial_y u_0 \equiv 0$ in $\overline{\R^2_+}$, then $u_0$ is a bounded function depending only of $x$, hence   $u_0(x,y)=c_1 x^{1-\alpha}+c_2$, and by the boundedness of $u_0$, we have $c_1=0$.  

From now on, we assume that $u_0$ is strictly monotone in $y$. Since $u_0$ is bounded, this implies that $u_0(0,y)$ has limits when $y \rightarrow \pm \infty$.   We now reach a contradiction by invoking the following theorem, proved in \cite{CS} and relying on a Hamiltonian estimate, proved in \cite{CS}. 
\begin{thm}(\cite{CS})\label{hamil}
Let $u_0$ be a bounded weak solution of \eqref{1D} such that  
$$\lim_{y \rightarrow \pm \infty}u_0(0,y)=\alpha^{\pm}.$$
Assume in addition that $u_0(0,y)$ is strictly monotone in $y$. Then 
$$G(\alpha^+)=G(\alpha^-)$$
where $G'=-f. $
\end{thm}

From Theorem \ref{hamil}, we deduce that the nonlinearity has to be balanced, i.e. 
$$\int_{\alpha^-}^{\alpha^+}f(x)\,dx=0,$$
hence a contradiction with $f \geq 0$, unless $f\equiv0$ on the range of $(\a^-,\a^+)$. 
Hence $u_0$ is actually a bounded weak solution of    
\begin{equation}
\left \{
\begin{matrix} 
{\rm div}\, (x^\a \nabla u_0)=0 \qquad 
{\mbox{ on $\R^{2}_+ $}} 
\\
-x^\a \partial_x u_0 = 0. 
\qquad{\mbox{ on $\partial\R^{2}_+ $.}}\end{matrix}\right . \end{equation}
Since $u_0$ has zero conormal  derivative on the boundary, one can reflect it oddly to obtain a new function (still denoted $u_0$) satisfying weakly div$(|x|^\alpha \nabla u_0)=0$ in $\R^2$. Applying Proposition 2.6 in \cite{CSS} and using the fact that $u_0$ is bounded, one gets that $u_0$ has the form 
$$u_0(x,y)=c_1(y)x^{1-\alpha}+c_2(y)$$
for some functions $c_1,c_2: \R \to \R$. Since $u_0$ is bounded, this gives $c_1 \equiv 0$. Therefore, the function $u_0$, which depends only on $y \in \R$ satisfies $u_0 ''=0$, giving that $u_0$ is constant. 
\begin{remark}
In the range $\alpha \in (-1,0)$, one can give a shorter proof using directly Theorem \ref{liouville}.  Indeed, applying this theorem to $u_0$ and taking $\varphi \equiv 1$, one just needs to check the energy bound \eqref{intR2},  

$$\int_{B_R^+} x^\alpha u_0^2 \leq C \int_0^R x^\alpha dx \int_{-R}^R dy = 2 C R^{1+\alpha} R \leq C_* R^2$$
for $R>1$, hence the result that $u_0$ is constant. 
\end{remark}

\section*{Acknowledgments}
 The first author would like to thank the hospitality of Laboratoire Poncelet where part of this paper has been done. The second author is supported by the ANR project "PREFERED".   
 
\bibliographystyle{alpha} 
\bibliography{bibliofile}

\bigskip\bigskip\bigskip\bigskip

 {\em LD} -- 
LAMFA, UMR CNRS 6140, Universit\'e Picardie Jules Verne, 33 rue St Leu, 80039 Amiens, France.  

{\tt ldupaigne@math.cnrs.fr}
\medskip 

{\em YS} --  
Universit\'e Aix-Marseille 3, Paul C\'ezanne -- 
LATP and Laboratoire Poncelet, Moscow, Russia
Marseille, France. 
 
{\tt sire@cmi.univ-mrs.fr}

\end{document}